\documentclass[10pt]{amsart}
\usepackage{amssymb,amsmath,amsthm,amsfonts}

\newtheorem{theorem}{Theorem}[section]

\newtheorem{definition}[theorem]{Definition}
\newtheorem{proposition}[theorem]{Proposition}
\newtheorem{corollary}[theorem]{Corollary}

\newtheorem{example}[theorem]{Example}

\begin{document}

\title{On Baire classification of strongly separately continuous functions}

\author{Olena Karlova}

\maketitle

\begin{abstract}
We investigate strongly separately continuous functions on a product of topological spaces and prove that if $X$ is a countable product of real lines, then there exists a strongly separately continuous function $f:X\to\mathbb R$ which is not Baire measurable. We show that if  $X$ is a product of normed spaces $X_n$,  $a\in X$ and $\sigma(a)=\{x\in X:|\{n\in\mathbb N: x_n\ne a_n\}|<\aleph_0\}$ is a subspace of $X$, equipped with the Tychonoff topology, then for any open set $G\subseteq \sigma(a)$ there is a strongly separately continuous function $f:\sigma(a)\to \mathbb R$ such that the discontinuity point set of $f$ is equal to~$G$.
\end{abstract}

\section{Introduction.}

In 1998 Omar Dzagnidze~\cite{Dzagnidze} introduced a notion of  a strongly separately continuous function $f:\mathbb R^n\to\mathbb R$. Namely, he calls a function $f$ {\it strongly separately continuous} at a point $x^0=(x_1^0,\dots,x_n^0)\in\mathbb R^n$ if the equality
\begin{gather*}
 \lim\limits_{x\to x^0}|f(x_1,\dots,x_k,\dots,x_n)-f(x_1,\dots,x_k^0,\dots,x_n)|=0
\end{gather*}
 holds for every $k=1,\dots,n$. Dzagnidze proved that a function \mbox{$f:\mathbb R^n\to\mathbb R$} is strongly separately continuous at $x^0$ if and only if $f$ is continuous at $x^0$.

Extending these investigations, J.~\v{C}in\v{c}ura, T.~\v{S}al\'{a}t and T.~Visnyai \cite{CSV} consider strongly separately continuous functions defined on the space $\ell_2$ of sequences $x=(x_n)_{n=1}^\infty$ of real numbers such that $\sum\limits_{n=1}^\infty x_n^2<+\infty$ endowed with the standard metric $d(x,y)=(\sum\limits_{n=1}^\infty (x_n-y_n)^2)^{1/2}$. In particular, the authors gave an example of a strongly separately continuous everywhere discontinuous function $f:\ell_2\to\mathbb R$.

Recently, T.~Visnyai in \cite{TV} continued to study properties of strongly separately continuous functions on $\ell_2$ and constructed a strongly separately continuous function $f:\ell_2\to\mathbb R$ which belongs to the third Baire class and is not quasi-continuous at every point. Moreover, T.~Visnyai gave a sufficient condition for strongly separately continuous function to be continuous on $\ell_2$.

In this paper we study strongly separately continuous functions defined on a subspaces of a product $\prod\limits_{t\in T} X_t$ of topological spaces $X_t$ equipped with the Tychonoff topology of pointwise convergence.  We show that if $X$ is a product of a sequence $(X_n)_{n=1}^\infty$ of topological spaces $X_n$, $a\in X$ and $\sigma(a)=\{x\in X:|\{n\in\mathbb N: x_n\ne a_n\}|<\aleph_0\}$ is a subspace of $X$, equipped with the Tychonoff topology, then every strongly separately continuous function $f:\sigma(a)\to\mathbb R$ belongs to the first stable Baire class. Moreover, we prove that if $X$ is a countable product of real lines, then there exists a strongly separately continuous function $f:X\to\mathbb R$ which is not Baire measurable. In the last section we show that if  $X$ is a product of normed spaces, then for any open set $G\subseteq \sigma(a)$ there is a strongly separately continuous function $f:\sigma(a)\to \mathbb R$ such that the discontinuity point set of $f$ is equal to~$G$.

 \section{Strongly separately continuous functions and $\mathcal S$-open sets}

Let $X=\prod\limits_{t\in T} X_t$ be a product of a family of sets $X_t$ with $|X_t|>1$ for all $t\in T$.
If $S\subseteq S_1\subseteq T$, $a=(a_t)_{t\in T}\in X$, $x=(x_t)_{t\in S_1}\in \prod\limits_{t\in S_1}X_t$, then we denote by $a_S^x$ a point $(y_t)_{t\in T}$, where
$$
y_t=\left\{\begin{array}{ll}
             x_t, & t\in S, \\
             a_t, & t\in T\setminus S.
           \end{array}
\right.
$$
In the case $S=\{s\}$ we shall write $a_s^x$ instead of $a_{\{s\}}^x$.

If $n\in\mathbb N$, then we set
\begin{gather*}
\sigma_n(x)=\{y=(y_t)_{t\in T}\in X: |\{t\in T:y_t\ne x_t\}|\le n\}
\end{gather*}
and
$$
\sigma(x)=\bigcup\limits_{n=1}^\infty \sigma_n(x).
$$
Each of the sets of the form $\sigma(x)$ for an $x\in X$ is called {\it a $\sigma$-product of the space $X$}.

We denote by $\tau$ the Tychonoff topology on a product $X=\prod\limits_{t\in T}X_t$ of topological spaces $X_t$. If $X_0\subseteq X$, then the symbol $(X_0,\tau)$ means the subspace $X_0$ equipped with  the Tychonoff topology induced from $(X,\tau)$.

If $X_t=Y$ for all $t\in T$ then the product $\prod\limits_{t\in T}X_t$ we also denote by $Y^{\mathfrak m}$, where $\mathfrak m=|T|$.

A set $E\subseteq \prod\limits_{t\in T} X_t$ is called {\it $\mathcal S$-open} if
$$
\sigma_1(x)\subseteq E
$$ for all $x\in E$.

Let $\mathcal S(X)$ denote the collection of all $\mathcal S$-open subsets of $X$. We notice that $\mathcal S(X)$ is a topology on $X$. We will denote by  $(X,\mathcal S)$ the product $X=\prod\limits_{t\in T}X_t$ equipped with the topology $\mathcal S(X)$.

The next properties follow easily from the definitions.

\begin{proposition}\label{prop:simple_prop_s-open} Let $X=\prod\limits_{t\in T}X_t$, $|X_t|>1$ for all $t\in T$  and $E\subseteq X$. Then
\begin{enumerate}
  \item \label{S-Clopen} $E\in\mathcal S(X)$ if and only if $X\setminus E \in \mathcal S(X)$;

  \item\label{pr:union-snow}  $E\in\mathcal S(X)$ if and only if $E=\bigcup\limits_{x\in E}\sigma(x)$;

  \item\label{prop:componenta} if $x\in X$, then $\sigma(x)$ is the smallest $\mathcal S$-open set which contains $x$;

  \item\label{s_open_dense}  if $E\in \mathcal S(X)$, then $E$ is dense in $(X,\tau)$.

  \item\label{non-trivial}  there exists a non-trivial $\mathcal S$-open subset of $X$ if and only if  $|T|\ge\aleph_0$.
\end{enumerate}
\end{proposition}

It follows from Proposition~\ref{prop:simple_prop_s-open} that $\sigma$-products of two distinct points of $\prod\limits_{t\in T}X_t$ either coincide, or does not intersect. Consequently, the family of all $\sigma$-products of an arbitrary $\mathcal S$-open set $E\subseteq \prod\limits_{t\in T}X_t$ generates a partition of $E$ on  mutually disjoint $\mathcal S$-open sets, which we will call {\it $\mathcal S$-components of $E$.}

\begin{definition}\label{def:sep-cont}
  {\rm  Let $(X_t:t\in T)$ be a family of topological spaces, $Y$ be a topological space and let $E\subseteq \prod\limits_{t\in T}X_t$ be an $\mathcal S$-open set. A mapping $f:E\to Y$ is said to be {\it separately continuous at a point $a=(a_t)_{t\in T}\in E$ with respect to the $t$-th variable} provided that the mapping $g:X_t\to Y$ defined by the rule $g(x)=f(a_t^x)$ for all $x\in X_t$ is continuous at the point $a_t\in X_t$.}
\end{definition}

\begin{definition}
  {\rm Let $E\subseteq\prod\limits_{t\in T}X_t$ be an $\mathcal S$-open set, $\mathcal T$ be a topology on $E$ and let $(Y,d)$ be a metric space. A mapping $f:(E,\mathcal T)\to Y$ is called {\it strongly separately continuous at a point $a\in E$ with respect to the $t$-th variable} if
  $$
\lim_{x\to a}d(f(x),f(x_{t}^a))=0.
  $$}
\end{definition}

\begin{definition}
   {\rm A mapping $f:E\to Y$ is
    \begin{itemize}
      \item {\it (strongly) separately continuous at a point $a\in E$}, if $f$ is (strongly) separately continuous at  $a$ with respect to each variable $t\in T$;

    \item {\it (strongly) separately continuous on the set $E$}, if $f$ is (strongly) separately continuous  at every point $a\in E$ with respect to each variable $t\in T$.
        \end{itemize} }
  \end{definition}

\begin{theorem}\label{prop:strong-top-s}
 Let $E\subseteq\prod\limits_{t\in T}X_t$ be an $\mathcal S$-open set and $(Y,d)$ be a metric space. A mapping \mbox{$f:(E,\mathcal S)\to Y$} is continuous if and only if $f:(E,\mathcal T)\to Y$ is strongly separately continuous for an arbitrary topology $\mathcal T$ on $E$.
\end{theorem}

\begin{proof}
{\it Necessity.} Fix a topology $\mathcal T$ on $E$ and  consider the partition $(\sigma(x_i):i\in I)$ of the set $E$ on $\mathcal S$-components $\sigma(x_i)$. We notice that $f|_{\sigma(x_i)}=y_i$, where $y_i\in Y$ for all $i\in I$, since $f$ is continuous on $(E,\mathcal S)$. Let $a=(a_t)_{t\in T}\in E$ and $t\in T$. If $x\in E$, then $x\in\sigma(x_i)$ for some $i\in I$. Moreover, $x_{t}^a\in\sigma(x_i)$. Then $f(x)=f(x_{t}^a)=y_i$. Hence, $d(f(x),f(x_{t}^a))=0$ for all $x\in E$. Hence, $f$ is strongly separately continuous on $(E,\mathcal T)$.

{\it Sufficiency.} Put $\mathcal T=\mathcal S$. Fix $x_0\in E$ and show that $f$ is continuous at $x_0$ on $(E,\mathcal S)$. Let $x_0\in\sigma(x_i)$ for some $i\in I$. Let us observe that $x\to x_0$ in $(E,\mathcal S)$ if and only if $x\in\sigma(x_0)$. Since $f$ is strongly separately continuous at $x_0$ and $\sigma(x_0)=\sigma(x_i)$,  we have $d(f(x),f(x_{t}^{x_0}))=0$ for all $x\in\sigma(x_i)$ and $t\in T$. Consequently, $f(x)=f(x_0)$ for all $x\in \sigma(x_i)$. Since the set $\sigma(x_i)$ is open in $(E,\mathcal S)$, $f$ is continuous at $x_0$.
\end{proof}

Let $(\sigma_i:i\in I)$ be a partition of $X=\prod\limits_{t\in T}X_t$ on $\mathcal S$-components and
$f:\prod\limits_{t\in T}X_t\to\mathbb R$ be a function such that $f|_{\sigma_i}\equiv {\rm const}$ for all $i\in I$. Theorem~\ref{prop:strong-top-s} implies that $f$ is strongly separately continuous on $(X,\tau)$, since for every $i\in I$ the set $\sigma_i$ is clopen in $(X,\mathcal S)$. The next example shows that it is not so in the case $f|_{\sigma_i}$ is a continuous functions on $(\sigma_i,\tau)$ for every $i\in I$.

\begin{example}
  Let $X=\mathbb R^{\aleph_0}$,  $(\sigma_i:i\in I)$ be a partition of $X$ on $\mathcal S$-components and let $\sigma(m)=\{x=(x_n)\in X:|\{n\in\mathbb N: x_n\ne m\}|<\aleph_0\}$ for all $m\in\mathbb N$. Consider a function
  $f:X\to\mathbb R$,
  $$
  f(x)=\left\{\begin{array}{ll}
                m\cdot(x_1+\dots+x_m), & \mbox{if}\,\, x\in \sigma(m),\\
                0, & \mbox{otherwise}.
              \end{array}
  \right.
  $$
  Then $f|_{\sigma_i}:(\sigma_i,\tau)\to\mathbb R$ is continuous for every $i\in I$, but $f$ is not strongly separately continuous at $x=0$.
\end{example}

\begin{proof}
 For every $m\in\mathbb N$ we put
  $$
  u^m=\Bigl(\mathop{\underbrace{\frac 1m,\dots,\frac 1m}}\limits_{m},m,m,\dots \Bigr).
  $$
  Then $u^m\in\sigma(m)$ and $u^m\to 0$ in $(X,\tau)$. Note that $f(u^m)=m$ and $f((u^m)_1^x)=m-1$. Therefore,
  $|f(u^m)-f((u^m)_1^x)|=1$ for all $m\in\mathbb N$. Consequently, $f$  is not strongly separately continuous at $x=0$ with respect to the first variable.
\end{proof}

\begin{theorem}\label{th:suff-cond}
  Let $E\subseteq \prod\limits_{t\in T}X_t$ be an $\mathcal S$-open subset of a product of topological spaces $X_t$, $(Y,d)$ be a metric space and let $f:(E,\tau)\to Y$ be a strongly separately continuous mapping at the point  $a=(a_t)_{t\in T}\in E$. Then   $f$ is continuous at the point $a$ if and only if
 \begin{gather}
  \forall \varepsilon>0\,\,\,\exists\, T_0\subseteq T, \,\,|T_0|<\aleph_0\nonumber\\
   \exists U \mbox{-- a neighborhood of\,\,\,} a \mbox{\,\,\,in\,\,\,} (E,\mathcal T) \,|\\\label{eq:suff-cond}
d(f(a),f(x_{T_0}^{a}))<\varepsilon \quad \forall\, x\in U.
\nonumber
 \end{gather}
\end{theorem}

\begin{proof} {\it Necessity}. Suppose $f$ is continuous at the point $a$ and $\varepsilon>0$. Take a basic neighborhood $U$ of $a$ such that $d(f(x),f(a))<\varepsilon$ for all $x\in U$ and put $T_0=\emptyset$. Then $x_{T_0}^a=x$, which implies condition (1).

{\it Sufficiency}.  Fix $\varepsilon>0$. Using the condition of the theorem we take a finite set $T_0\subseteq T$ and a neighborhood $U$ of $a$ in $(E,\tau)$ such that $$
  d(f(a),f(x_{T_0}^{a}))<\frac{\varepsilon}{2}
  $$
  for  every $x\in U$. If $T_0=\emptyset$, then  $d(f(x),f(a))<\varepsilon$ for all $x\in U$, which implies the continuity of $f$ at $a$. Now assume $T_0=\{t_1,\dots,t_n\}$.  Since $f$ is strongly separately continuous at $a$, for every $k=1,\dots,n$ we choose a neighborhood $V_k$ of the point $a$ such that
  $$
  d(f(x),f(x_{t_k}^{a}))<\frac{\varepsilon}{2n}
  $$
  for all $x\in V_k$. We take a basic neighborhood $W$ of $a$ such that
  $$
  W\subseteq U\cap\bigl(\bigcap\limits_{k=1}^n V_k\bigr).
  $$
Observe that $x_{\{t_1,\dots,t_k\}}^a\in W$ for every $k=1,\dots,n$ and  for every $x\in W$.  Then for all $x\in W$ we have
\begin{gather*}
  d(f(x),f(a))\le  d(f(x),f(x_{T_0}^a))+d(f(x_{T_0}^a),f(a))<\\ <d(f(x),f(x_{\{t_1\}}^a))+d(f(x_{\{t_1\}}^a),f(x_{\{t_1,t_2\}}^a))+\dots+\\+d(f(x_{\{t_1,\dots,t_{n-1}\}}^a),f(x_{\{t_1,\dots,t_n\}}^a))+\frac{\varepsilon}{2}<\frac{\varepsilon}{2n}\cdot n+\frac{\varepsilon}{2}=\varepsilon.
\end{gather*}
Hence, $f$ is continuous at the point $a$.
\end{proof}

The following corollary generalizes the result of Dzagnidze~\cite[Theorem 2.1]{Dzagnidze}.

\begin{corollary}\label{cor:finite-dim-strong-cont}
  Let $E$ be an $\mathcal S$-open subset of a product $\prod\limits_{t\in T} X_t$ of topological spaces $X_t$, $|T|<\aleph_0$ and $(Y,d)$ be a metric space. Then any strongly separately continuous mapping \mbox{$f:(E,\tau)\to Y$} is continuous.
\end{corollary}

\begin{proof}
  Fix an arbitrary point $a\in E$ and a strongly separately continuous mapping \mbox{$f:(E,\tau)\to Y$}. For $\varepsilon>0$ we put $T_0=T$ and $U=E$. Then for all $x\in U$ we have $x_{T_0}^a=a$ and consequently
  $$
  d(f(a),f(x_{T_0}^a))=0<\varepsilon.
  $$
  Hence, $f$ is continuous at the point $a$ by Theorem~\ref{th:suff-cond}.
\end{proof}

The proposition below shows that Corollary~\ref{cor:finite-dim-strong-cont} is not valid for a product of  infinitely many topological spaces.

\begin{proposition}\label{prop:everywhere_discont}
  Let $X=\prod\limits_{t\in T}X_t$ be a product of topological spaces $X_t$, where $|X_t|>1$ for every $t\in T$, let $|T|> \aleph_0$ and $(Y,d)$ be a metric space with $|Y|>1$. Then there exists a strongly separately continuous everywhere discontinuous mapping $f:(X,\tau)\to Y$.
\end{proposition}

\begin{proof}
  Fix $x_0\in X$ and $y_1,y_2\in Y$, $y_1\ne y_2$. According to Proposition~\ref{prop:simple_prop_s-open}(\ref{non-trivial}), $\sigma(x_0)\ne\emptyset\ne X\setminus\sigma(x_0)$.  Set $f(x)=y_1$ if $x\in\sigma(x_0)$ and $f(x)=y_2$ if $x\in X\setminus \sigma(x_0)$. We prove that $f$ is everywhere discontinuous on $X$. Indeed, let $a\in X$  and $f(a)=y_1$. Take an open neighborhood $V$ of $y_1$ such that $y_2\not\in V$. If $U$ is an arbitrary neighborhood of $a$ in $(X,\tau)$, then there is $x\in U\setminus\sigma(x_0)$ by Proposition~\ref{prop:simple_prop_s-open}~(\ref{s_open_dense}). Then $f(x)=y_2\not\in V$. Therefore, $f$ is discontinuous at $a$. Similarly one can show that $f$ is discontinuous at $a$ in the case $f(a)=y_2$.

   Since the set $\sigma(x_0)$ is clopen in $(X,\mathcal S)$, the mapping $f:(X,\mathcal S)\to Y$ is continuous. It remains to apply Theorem~\ref{prop:strong-top-s}.
\end{proof}

\section{Baire measurable strongly separately continuous functions}

Let $B_0(X,Y)$ be a collection of all continuous mappings $f:X\to Y$. Assume that the classes $B_\xi(X,Y)$ are already defined for all $0\le \xi<\alpha$, where $\alpha<\omega_1$. Then $f:X\to Y$ is {said to be} {\it of the $\alpha$-th Baire class}, $f\in B_\alpha(X,Y)$, if $f$ is a pointwise limit of a sequence of mappings $f_n\in B_{\xi_n}(X,Y)$, where $\xi_n<\alpha$. Denote
$$
\mathcal B(X,Y)=\bigcup\limits_{0\le\alpha<\omega_1}B_\alpha(X,Y).
$$
 We say that $f:X\to Y$ {\it is a Baire measurable mapping}, if $f\in \mathcal B(X,Y)$.

Let $0\le\alpha<\omega_1$, $X$ be a metrizable space, $Y$ is a topological space and let $Z$ be  a locally convex space. W.~Rudin~\cite{Ru} proved that every mapping $f:X\times Y\to Z$, which is continuous with respect to the first variable and is of the $\alpha$-th Baire class with respect to the second one, belongs to the $(\alpha+1)$-th  Baire class on $X\times Y$. The following proposition is an easy corollary of the Rudin Theorem.

\begin{proposition}\label{prop:cor-Rudin}
  Let $n\in\mathbb N$, $X_1,\dots,X_n$ be metrizable spaces and $Z$ be a locally convex space. Then every separately continuous mapping $f:\prod\limits_{i=1}^n X_i\to Z$ belongs to the $(n-1)$-th Baire class.
\end{proposition}

\begin{proof}
The assertion of the proposition is evident if $n=1$ and is exactly the Rudin Theorem if $n=2$. Now assume that the proposition is true for all $2\le k<n$ and prove it for $k=n$. Denote $X=\prod\limits_{i=1}^{n-1}X_i$. Then $f:X\times X_n\to Z$ belongs to the $(n-2)$-th Baire class with respect to the first variable by the inductive assumption and $f$ is continuous with respect to the second variable. Applying the Rudin Theorem we have $f\in B_{n-1}(X\times X_n,Z)$.
\end{proof}

The next result shows that the corollary of Rudin's Theorem is not valid for infinite products.

\begin{proposition}\label{ex:contr_Lebesgue}
  There exists a strongly separately continuous function $f:(\mathbb R^{\aleph_0},\tau)\to\mathbb R$ which is not Baire measurable.
\end{proposition}

\begin{proof}
Consider a partition $(\sigma_i:i\in I)$ of $\mathbb R^{\aleph_0}$ on $\mathcal S$-components $\sigma_i$. It is not hard to verify that $|I|=\mathfrak c$. Denote by $\mathcal F$ the collection of all functions $f:\mathbb R^{\aleph_0}\to\mathbb R$ such that $f|_{\sigma_i}={\rm const}$ for all $i\in I$. Then $|\mathcal F|=2^{|I|}=2^{\mathfrak c}$. Moreover, since $({\mathbb R}^{\aleph_0},\tau)$ is separable, $|B_0({\mathbb R}^{\aleph_0},\mathbb R)|=\mathfrak c$ and, consequently, $|\mathcal B(X,Y)|=\mathfrak c$. Hence, there exists $f\in\mathcal F\setminus\mathcal B({\mathbb R}^{\aleph_0},\mathbb R)$. Since for every $i\in I$ the set $\sigma_i$ is clopen in $({\mathbb R}^{\aleph_0},\mathcal S)$,  $f$ is continuous on $({\mathbb R}^{\aleph_0},\mathcal S)$. Then $f$ is strongly separately continuous on $({\mathbb R}^{\aleph_0},\tau)$ according to Proposition~\ref{prop:strong-top-s}.
\end{proof}

Let $1\le\alpha<\omega_1$. A mapping $f:X\to Y$ belongs to the {\it $\alpha$-th stable Baire class}, $f\in B_{\alpha}^d(X,Y)$, if there exists a sequence of mappings  $f_n\in B_{\alpha_n}(X,Y)$, where $\alpha_n<\alpha$, such that for every  $x\in X$ there exists $N\in\mathbb N$ such that $f_n(x)=f(x)$ for all $n\ge N$.

\begin{theorem}\label{prop:restriction_on_sigma}
Let $(X_n)_{n=1}^\infty$ be a sequence of topological spaces, $a\in \prod\limits_{n=1}^\infty X_n$, $E=\sigma(a)$   and let $f:(E,\tau)\to\mathbb R$ be a function.
\begin{enumerate}
  \item If $f$ is strongly separately continuous, then $f\in B_1^d(E,\mathbb R)$.

  \item If $f$ is separately continuous and  $X_n$ is metrizable for every $n\in\mathbb N$, then $f\in B_{\omega_0}^d(E,\mathbb R)$.
\end{enumerate}
\end{theorem}

\begin{proof}
 For every $n\in\mathbb N$ we put
$$
E_n=\prod\limits_{i=1}^n X_i\times\prod\limits_{i=n+1}^\infty\{a_i\},
$$
$g_n=f|_{E_n}$ and
$$
f_n(x)=g_n(x_1,\dots,x_n,a_{n+1},\dots)
$$ for all $x\in E$.
Clearly, $E=\bigcup\limits_{n=1}^\infty E_n$, $E_n\subseteq E_{n+1}$ and every space $(E_n,\tau)$ is homeomorphic to $(\prod\limits_{i=1}^n X_i,\tau)$.

If $f$ is strongly separately continuous, then by Theorem~\ref{cor:finite-dim-strong-cont} every $g_n$ is continuous on $E_n$. Then $f_n:(E,\tau)\to\mathbb R$ is a continuous extension of $g_n$.

In the second case $g_n\in B_{n-1}(E_n,Z)$ by Proposition~\ref{prop:cor-Rudin} for every $n$. It is not hard to verify that $f_n\in B_{n-1}(E,Z)$.

Now if $x\in E$, then there is $N\in\mathbb N$ such that $x\in E_n$ for all $n\ge N$. Therefore, $f_n(x)=f(x)$ for all $n\ge N$. Hence, $f\in B_1^d(E,\mathbb R)$ in the first case and $f\in B_{\omega_0}^d(E,\mathbb R)$ in the second one.
\end{proof}

\begin{proposition}
  Let $a=(0,0,\dots)\in \mathbb R^{\aleph_0}$, $E=\sigma(a)\subseteq \mathbb R^{\aleph_0}$ and $Y=[0,1]$. Then there exists a separately continuous function $f:E\to Y$ such that $f\not\in \bigcup\limits_{n=1}^\infty B_n((E,\tau),Y)$.
\end{proposition}

\begin{proof}
  For every $n\in\mathbb N$ we take a function $h_n\in B_{n+1}(\mathbb R, Y)\setminus B_{n}(\mathbb R, Y)$. According to the Lebesgue Theorem~\cite{Leb1} for every $n\in\mathbb N$ there exists a separately continuous function \mbox{$g_n:\mathbb R^{n+2}\to Y$} such that $$g_n(\mathop{\underbrace{x,x,\dots,x}}\limits_{n+2})=h_n(x)$$ for each $x\in\mathbb R$. Evidently, $g_n$ is not of the $n$-th Baire class on $\mathbb R^{n+2}$.

Let $\varphi:\mathbb R\to Y$ be any continuous function such that $\{0\}=\varphi^{-1}(0)$.
 For $n\in\mathbb N$ we consider a function $f_n:E\to Y$,
  $$
  f_n(x_1,\dots,x_n,\dots)=\varphi(x_{n+2})\cdot g_n(x_1,\dots,x_{n+2}).
  $$
  Then the function $f_n:E\to Y$ is separately continuous as the product of two separately continuous functions. Moreover,
  $$
  f_n|_{E_{n+2}}\not\in B_{n}(E_{n+2}, Y)
  $$
  for every $n\in\mathbb N$, where
  $$
  E_n=\mathbb R^n\times\{0\}\times\{0\}\times\dots.
  $$
  For every $x\in E$ we put
  $$
  f(x)=\sum\limits_{n=1}^\infty \frac{1}{2^n}f_n(x).
  $$

Observe that $f:E\to\mathbb R$ is separately continuous as the sum of the uniformly convergent series of separately continuous functions.

It remains to show that $f\not \in\bigcup\limits_{n=1}^\infty B_{n}(E,Y)$. Assume to the contrary that $f\in B_n(E,Y)$ for some $n\in\mathbb N$. Then $f|_{E_{n+2}}\in B_n(E_{n+2},Y)$. Notice that
$$
f|_{E_{n+2}}=\sum\limits_{k=1}^n \frac{1}{2^k}f_k|_{E_{n+2}},
$$
since $f_k|_{E_{n+2}}=0$ for all $k\ge n+1$. Denote
$$
g=\sum\limits_{k=1}^{n-1} \frac{1}{2^k}f_k|_{E_{n+2}}.
$$
Then we have $g\in B_n(E_{n+2},Y)$, since
$$
f_k|_{E_{n+2}}\in B_{k+1}(E_{n+2},Y)\subseteq B_n(E_{n+2},Y)
$$
for every $k=1,\dots,n-1$. Therefore,
$$
f_n|_{E_{n+2}}=(f|_{E_{n+2}}-g)\in B_n(E_{n+2},Y),
$$
which implies a contradiction.
\end{proof}

\section{Discontinuities of strongly separately continuous mappings}

For a mapping $f$ between spaces $X$ and $Y$ we denote the set of all points of continuity of $f$ by   $C(f)$ and let $D(f)=X\setminus C(f)$.

\begin{theorem}\label{th:inverse-for-normed}
Let $X=\prod\limits_{n=1}^\infty X_n$ be a product of normed spaces $(X_n,\|\cdot\|_n)$ and let $a\in X$. Then for any open set $G\subseteq (\sigma(a),\tau)$ there exists a strongly separately continuous function $f:(\sigma(a) ,\tau)\to \mathbb R$ such that $D(f)=G$.
\end{theorem}

\begin{proof}
  Without loss of generality we may assume that $a=(0,0,\dots)$. For every $n\in\mathbb N$ we consider a norm $\|\cdot\|_n$ on the space $X_n$ which generates its topological structure and let $d$ be a bounded metric on $X$ which generates the Tychonoff topology $\tau$. Denote $X_0=(\sigma(a),\tau)$ and $F=X_0\setminus G$. For every $x=(x_n)_{n\in\mathbb N}\in X_0$ put
\begin{gather*}
   \varphi(x)=d(x,F),\,\,\mbox{if}\,\, F\ne\emptyset,\\
   \varphi(x)=1,\,\,\,\mbox{if}\,\, F=\emptyset,\\
   g(x)=\exp(-\sum\limits_{n=1}^\infty \|x_n\|_n),\\
   f(x)=\varphi(x)\cdot g(x).
\end{gather*}

We prove that $F\subseteq C(f)$. Indeed, if $x^0\in F$ and $(x^m)_{m=1}^\infty$ is a convergent to $x^0$ sequence in $X_0$, then $\lim\limits_{m\to\infty}\varphi(x^m)\cdot g(x^m)=0$, since $\lim\limits_{m\to\infty}\varphi(x^m)=\varphi(x^0)=0$ and $|g(x^m)|\le 1$ for every  $m$. Hence, $\lim\limits_{m\to\infty}f(x^m)=0=f(x^0)$.

Fix an arbitrary $x^0\in G$ and show that $x^0\in D(f)$. For every $m\in\mathbb N$ we choose $x_m\in X_m$ with $\|x_m\|_m=\ln 2+\|x_m^0\|_m$ and set $$
x^m=(x_1^0,x_2^0,\dots,x_{m-1}^0,x_m,x_{m+1}^0,\dots).
$$
 Clearly, $x^m\to x^0$ in $X_0$. For every $m\in\mathbb N$ we have
\begin{gather*}
  g(x^m)-g(x^0)=\exp(-\sum\limits_{n=1}^\infty\|x_n^0\|_n)(\exp(\sum\limits_{n=1}^\infty\|x_n^0\|_n-\sum\limits_{n=1}^\infty\|x_n^m\|_n)-1)=\\
  =g(x^0)(\exp(-\ln 2)-1)=-\frac 12 g(x^0).
\end{gather*}
Therefore, $g(x^m)=\frac 12 g(x^0)$ and
\begin{gather*}
  f(x^m)-f(x^0)=\varphi(x^m)g(x^m)-\varphi(x^0)g(x^0)= g(x^0)(\frac 12\varphi(x^m)-\varphi(x^0))
\end{gather*}
for all $m\in\mathbb N$.  Then
 $$
 \lim\limits_{m\to\infty}(f(x^m)-f(x^0))=-\frac 12 \varphi(x^0)\cdot g(x^0)<0.
 $$
Hence, $f$ is discontinuous at $x^0$. Consequently, $D(f)=G$.

It remains to check that $f$ is strongly separately continuous on $X_0$. Evidently, $f$ is strongly separately continuous on the set $C(f)=F$. Fix
 $x^0\in G$, $k\in\mathbb N$ and an arbitrary convergent to $x^0$ sequence $(x^m)_{m=1}^\infty$ in $X_0$. For every $m$ we put $y^m=(x^m)_{\{k\}}^{x^0}$. Since $G$ is open and $y^m\to x^0$, we may suppose that $x^m, y^m\in G$ for every $m$.  We note  that
 \begin{gather*}
   f(x^m)-f(y^m)=g(x^m)(\varphi(x^m)-\varphi(y^m)) +\varphi(y^m)(g(x^m)-g(y^m)) =\\=g(x^m)(\varphi(x^m)-\varphi(y^m))+\varphi(y^m)g(y^m)(\exp(\|x_k^0\|_k-\|x_k^m\|_k)-1).
 \end{gather*}
It follows from the inequality
 \begin{gather*}
   \exp(-\|x_k^0-x_k^m\|)\le \exp(\|x_k^0\|_k-\|x_k^m\|_k)\le\exp(\|x_k^0-x_k^m\|)
 \end{gather*}
that
 \begin{gather*}
   \lim\limits_{m\to\infty} (\exp(\|x_k^0\|_k-\|x_k^m\|_k)-1)=0.
 \end{gather*}
Taking into account that $\varphi$ and $g$ are bounded and
$\lim\limits_{m\to\infty}\varphi(x^m)=\lim\limits_{m\to\infty}\varphi(y^m)=\varphi(x^0)$, we obtain that
\begin{gather*}
  \lim\limits_{m\to\infty} (f(x^m)-f(y^m))=0.
\end{gather*}
Hence, $f$ is strongly separately continuous on~$X_0$.
\end{proof}

\end{document}